\documentclass[12pt]{amsart}
\usepackage[top=1in,bottom=1in, left=1 in, right= 1 in, includehead,includefoot]{geometry}

\usepackage{amsthm}
\usepackage{amsmath,amssymb}
\usepackage{dsfont}
\usepackage{color,hyperref}
\usepackage[numbers,sort&compress]{natbib}
\usepackage{enumerate}
\usepackage[utf8]{inputenc}

\newtheorem{theorem}{Theorem}[section]

\newtheorem{lemma}[theorem]{Lemma}
\newtheorem{proposition}[theorem]{Proposition}

\newtheorem{assumption}[theorem]{Assumption}

\theoremstyle{definition}

\theoremstyle{remark}
\newtheorem{remark}[theorem]{Remark}
\theoremstyle{remark}
\newtheorem{example}{Example}[section]

\numberwithin{equation}{section}

\def\fddto{\xrightarrow{\textit{f.d.d.}}}
\newcommand{\ind}{{\bf 1}}

\def\inddd#1{{\ind}_{\left\{#1\right\}}} 
\newcommand{\proba}{\mathbb P}
\newcommand{\esp}{{\mathbb E}}

\newcommand{\inv}{^{-1}}
\newcommand{\cov}{{\rm{Cov}}}

\newcommand{\eqnh}{\begin{eqnarray*}}
\newcommand{\eqne}{\end{eqnarray*}}
\newcommand{\eqnhn}{\begin{eqnarray}}
\newcommand{\eqnen}{\end{eqnarray}}
\newcommand{\equh}{\begin{equation}}
\newcommand{\eque}{\end{equation}}

\def\summ#1#2#3{\sum_{#1 = #2}^{#3}}
\def\prodd#1#2#3{\prod_{#1 = #2}^{#3}}
\def\sif#1#2{\sum_{#1=#2}^\infty}

\newcommand{\eqd}{\stackrel{d}{=}}

\newcommand{\PPP}{{\rm PPP}}
\def\topp#1{^{(#1)}}

\def\nn#1{{\left\|#1\right\|}}

\def\abs#1{\left|#1\right|}
\def\sabs#1{|#1|}

\def\ccbb#1{\left\{#1\right\}} 

\def\pp#1{\left(#1\right)}

\def\mmid{\;\middle\vert\;}

\def\floor#1{\left\lfloor #1 \right\rfloor}
\def\sfloor#1{\lfloor #1 \rfloor}

\def\vv#1{{\boldsymbol #1}}

\def\qmand{\quad\mbox{ and }\quad}

\def\qmwith{\quad\mbox{ with }\quad}
\def\mfa{\mbox{ for all }}

\def\wt#1{\widetilde{#1}}

\def\what#1{\widehat{#1}}

\def\limn{\lim_{n\to\infty}}

\def\limsupn{\limsup_{n\to\infty}}

\def\weakto{\Rightarrow}

\def\R{{\mathbb R}}

\renewcommand{\S}{{\mathbb S}}
\def\N{{\mathbb N}} 
\def\B{{\mathbb B}}
\newcommand{\M}{{\mathbb M}}

\newcommand{\calE}{{\mathcal E}}
\newcommand{\calF}{{\mathcal F}}

\newcommand{\calH}{{\mathcal H}}

\newcommand{\calM}{\mathcal{M}}

\newcommand{\calS}{\mathcal {S}}
\newcommand{\calX}{\mathcal{X}}

\renewcommand{\d}{{\rm d}}
\newcommand{\dd}{{\mathsf d}}
\newcommand{\MM}{{\mathsf M}}

\date{}

\def\R{\mathbb{R}}

\def\N{\mathbb{N}}

\def\1{\mathds{1}}
\renewcommand{\i}{{\rm i}}
\newcommand{\odd}{{\rm~is~odd}}
\newcommand{\even}{{\rm~is~even}}

\title{A family of log-correlated Gaussian processes}
\author{Yizao Wang}
\address
{
Yizao Wang\\
Department of Mathematical Sciences\\
University of Cincinnati\\
2815 Commons Way\\
Cincinnati, OH, 45221-0025, USA.
}
\email{yizao.wang@uc.edu}

\begin{document}\sloppy
\begin{abstract}
A family of log-correlated Gaussian processes indexed by metric spaces is introduced, when the metric is conditionally negative definite. These processes arise as the limit of bi-fractional Brownian motions indexed by $(H,K)$ scaled by $K^{-1/2}$ as $K\downarrow 0$ with $H\in(0,1/2]$ fixed. When the metric is in addition a measure definite kernel, stochastic-integral representations of the generalized processes when evaluated at a test function are provided. The introduced processes are also shown to be the scaling limits of certain aggregated models.
\end{abstract}
\maketitle
\section{Overview}

Log-correlated Gaussian processes have received extensive attentions recently in probability theory and related fields. Such processes defined on a locally compact metric space $(\MM,\dd)$ are generalized processes with covariance kernel in the form of
\[
\Gamma(x,y) = \log\frac C{\dd(x,y)} + g(x,y), \quad x,y\in\MM,
\]
where $g(x,y)$ is a continuous function on $\MM\times\MM$. The corresponding Gaussian process is not point-wise defined since $\Gamma(x,x) = \infty$. In the seminal work, \citet{kahane85chaos} established a general theory of building the Gaussian multiplicative chaos starting from such kernels. Many developments have followed recently. See the survey of \citet{rhodes14gaussian} and the forthcoming monograph by \citet{berestycki24gaussian}.  
In particular, many efforts have been devoted to specific log-correlated Gaussian processes, among which the most prominent example has been the Gaussian free fields \citep{berestycki24gaussian,sheffield07gaussian}. At the same time, log-correlated Gaussian processes and Gaussian multiplicative chaos have been shown to arise from limit theorems in random matrix theory~\citep{hughes01characteristic,chhaibi19circle,webb15characteristic} and mathematical finance~\citep{bacry03log,neuman18fractional,forde22riemann}. Most of the time, $\MM$ is either a domain of $\R^n$, or the circle $\S^1$. See also the survey by \citet{duplantier17log}.

The goal of this paper is to introduce a family of log-correlated Gaussian processes indexed by metric spaces with a covariance function explicitly determined by the metric $\dd$ and a parameter $H\in(0,1/2]$. For Brownian motion, Paul L\'evy already investigated the question of defining Brownian motion indexed by other spaces than $\R$~\citep{levy65processus,molchan67some,gangolli67positive}. Later on, the question of defining fractional Brownian motions indexed by  other spaces including 
Euclidean spaces $\R^d$, Euclidean spheres $\S^d$, and hyperbolic spaces $\mathbb H^d$ with Hurst index $H\in(0,1/2)$ has been addressed, and these processes indexed by metric spaces can be characterized as centered Gaussian processes $G^H$ with $\esp(G^H(x) - G^H(y))^2 = \dd^{2H}(x,y)$ and $G^H(o) = 0$ for some $o\in\MM$~\citep{istas06fractional,istas05spherical,istas12manifold,cohen12stationary}. 
A closely related question is the existence of the so-called {\em bi-fractional Brownian motions} indexed by metric spaces. A bi-fractional Brownian motion is a Gaussian process with two parameters that we let $H,K$ denote throughout. We shall see that for a large family of examples with $H\in(0,1/2]$, a log-correlated Gaussian process arises as the simple limit of bi-fractional Brownian motions with parameters $(H,K)$ normalized by $K^{-1/2}$ as $K\downarrow 0$. This, actually, follows from a simple calculation that we now explain.
\subsection{From the bi-fractional Brownian motion to a log-correlated Gaussian process}  
We let $G^{H,K} = (G^{H,K}(x))_{x\in\MM}$ denote the bi-fractional Brownian motion  indexed by a metric space $\MM$, and $H>0,K>0$ are two parameters that we shall discuss in a moment. This is a centered Gaussian process with covariance function
\equh\label{eq:cov G}
\cov\pp{G^{H,K}(x),G^{H,K}(y)} = \frac1{2^K}\pp{\pp{\dd^{2H}(o,x)+\dd^{2H}(o,y)}^K - \dd^{2HK}(x,y)}, x,y\in\MM,
\eque
where $o\in\MM$ is fixed and plays the role of the origin.  
Bi-fractional Brownian motion was first introduced by \citet{houdre03example} with $\MM = \R$ (and in this case $G^{H,K}$ is well-defined with $H\in(0,1], K\in(0,1]$ and $o = 0$). In particular, $G^{H,1}$ is just a fractional Brownian motion with Hurst index $H\in(0,1]$ ($H=1$ is a degenerate case).  
The exact range of legitimate  $(H,K)$ for the existence of $G^{H,K}$ remains an open question \citep{talarczyk20bifractional} even when $\MM = \R$ and $\dd$ is the Euclidean metric; see also \citet{lifshits15bifractional}. 

Since then, bi-fractional Brownian motions indexed by various metric spaces have been studied; see \citet{ma24bifractional}. In particular, when setting $K=1$ one always recover a fractional Brownian motion indexed by the same metric space with Hurst index $H$. 
It is important to keep in mind that the legitimate choice of $(H,K)$ depends on $(\MM,\dd)$. As a general setup we assume that $\dd$ is conditionally negative definite (a.k.a.~of negative type), and in this case it is easy to check that the function on the right-hand side of \eqref{eq:cov G} is positive definite and hence $G^{H,K}$ is well-defined with 
\[
H\in(0,1/2]\qmand K\in(0,1].
\]
 See \citet[Theorem 1.1, (vii)]{ma24bifractional}. In particular, $G^{H,1}$ is a fractional Brownian motion with Hurst index $H$ and $G^{1/2,1}$ is a Brownian motion.  
 
 Then, with $H$ fixed we have
\equh\label{eq:limit K}
\lim_{K\downarrow 0}\frac1K\cov\pp{G^{H,K}(x,y)} = \log\frac{\dd^{2H}(o,x)+\dd^{2H}(o,y)}{\dd^{2H}(x,y)}=:\Gamma\topp {H}(x,y), x,y\in\MM.
\eque
This simple calculation deserves to be stated in a proposition.
\begin{proposition}
If $(\MM,\dd)$ is a metric space and $\dd$ is of negative type, then for all $H\in(0,1/2]$, $\Gamma\topp{H}$ in \eqref{eq:limit K} is a covariance kernel for a generalized Gaussian process indexed by $\MM$. 
\end{proposition}
We let $G^H$ denote the so-introduced Gaussian process, and name it as the {\em log-correlated bi-fractional Brownian motion} with parameter $H\in(0,1/2]$. 

Our first main contribution is to provide  a stochastic-integral representation of $G^H$. Strictly speaking, we represent $G^H(f)$ (when evaluated at a suitable test function $f:\MM\to\R$) in the form of $G^H(f) = \int_\MM f(x)M(\d x)$ where $M$ is a suitable Gaussian random measure.  
For this purpose, we shall impose a further condition on $\dd$ being of negative type. 

It has been recently shown that when $\dd$ is a {\em measure definite kernel}, a notion from \citet{robertson98negative}, not only a fractional Brownian motion indexed by $\MM$ with Hurst index $H\in(0,1/2)$ can be defined, but also it has a stochastic-integral representation of Chentsov-type \citep{samorodnitsky94stable}. Our results are closely related to this unified framework for fractional Brownian motions indexed by metric spaces \citep{fu20stable}. The metric $\dd$ being a measure definite kernel is also our key assumption here.

\begin{assumption}\label{assump:1}
The metric $d$ on $\MM$ is a {\em measure definite kernel}, in the sense that there exists a measure space $(E,\calE,\mu)$ and a family of sets $\{A_x\}_{x\in \MM}\subset\calE$ such that $\mu(A_x)<\infty$ for all $x\in\MM$ and 
\equh\label{eq:MDK}
\dd(x,y) = \mu(A_x\Delta A_y), \mfa x,y \in \MM,
\eque
where $A\Delta B = (A\setminus B)\cup (B\setminus A)$.
\end{assumption}

Recall that if $\dd$ is of negative type then there exists a Gaussian process $(G(x))_{x\in\MM}$ such that $\esp(G(x)-G(y))^2 = \dd(x,y)$.  Every measure definite kernel is of negative type, but the opposite is not true. It has been well-known that if $\dd$ is of negative type, then so is $\dd^\beta$ for all $\beta\in(0,1)$. A brief review on metrics of negative type, measure definite kernels and their induced Gaussian processes can be found in \citet{istas12manifold}. It has been recently shown that if $\dd$ is a measure definite kernel, then so is $\dd^\beta$ for all $\beta\in(0,1)$ \citep{fu20stable}.

The first main result is the following.
\begin{theorem}\label{thm:0}
 If Assumption \ref{assump:1} is satisfied, then for $H\in(0,1/2]$, 
 \[
\Gamma\topp {H}(x,y) =  \log\frac{\mu^{2H}(A_x)+\mu^{2H}(A_y)}{\dd^{2H}(x,y)}, x,y\in\MM, \beta\in(0,1],
 \]
 is the covariance kernel of a generalized Gaussian process on $\MM$. 
More precisely, we have
 \equh\label{eq:Gamma_r}
 \Gamma\topp H(x,y) = \int_0^\infty \Gamma_r\topp H(x,y)\d r
 \eque
 where for each $r>0$, $\Gamma_r\topp H:\MM\times\MM\to\R$ is a positive definite function. 
\end{theorem}
We shall provide an explicit expression for $\Gamma_r\topp H$, which is essentially based on a stochastic-integral representation of $G(f)$, in Section \ref{sec:general}, Theorem \ref{thm:Kahane}.

\begin{remark}
Note that $\Gamma\topp H$ in \eqref{eq:Gamma_r} is a $\sigma$-positive kernel in the sense of \citet{kahane85chaos}, a property that is not easy to check in practice. 
The $\sigma$-positivity is a consequence of the metric $\dd$ being a measure definite kernel (see \eqref{eq:sigma positive}). 
Thanks to this property, the existence of the Gaussian multiplicative chaos determined by $\Gamma$ follows immediately from Kahane's framework. We leave the investigations of so-obtained Gaussian multiplicative chaos for future studies.
\end{remark}

Below we recall a few examples satisfying Assumption \ref{assump:1}. If there is a point $o\in\MM$ such that $\mu(A_o) = 0$, then the kernel simplifies to $\Gamma\topp H$ in \eqref{eq:limit K}.
We also provide an example that satisfies Assumption \ref{assump:1} but cannot be obtained as the limit of bi-fractional Brownian motions in Example \ref{example:1} below.

\begin{example}\label{example:d=1}
In the case $\MM = \R$, we take $A_x = [0,x]$ if $x\ge 0$ and $(-x,0]$ if $x<0$, and $\mu$ as the Lebesgue measure. Then, we have
\equh\label{eq:Gamma d=1}
\Gamma\topp H(x,y) = \log\frac{|x|^{2H}+|y|^{2H}}{|x-y|^{2H}}, x,y\in\R.
\eque
Obviously, when $H = 1/2$, $\Gamma(x,y) = 0$ if $x<0<y$, which means that the restrictions of the process to $(-\infty,0)$ and $(0,\infty)$ respectively are independent. In Sections \ref{sec:d=1} and \ref{sec:CLT} we shall focus on this example with $H= 1/2$ and restricted to $[0,\infty)$. 
\end{example}
\begin{example}
For $\MM = \R^n$, 
it is well-known that the Euclidean metric $\dd$ is a measure-definite kernel. Indeed, one can take $E:=\{(s,r):s\in\S^n, r\in(0,\infty)\}$ representing the space of all hyperplanes of $\R^n$ not including the origin, $A_x = \{(s,r):s\in\S^n, 0<r<\langle s,x\rangle\}$ the collection of all hyperplanes separating $0$ and $x$ in $\R^n$, and $\mu(\d s\d r) = \d s\d r$ ($\d s$ is the Lebesgue measure on $\S^n$, $\d r$ is the Lebesgue measure on $\R^1$, such that \eqref{eq:MDK} holds. In fact, this representation of the measure definite kernel is at the heart of earlier Chentsov representation for fractional L\'evy Brownian fields \citep[Chapter 8.3]{samorodnitsky94stable}. So the covariance kernel is
\[
\Gamma\topp H(x,y) =\log\frac{\nn {x}^{2H}+\nn{y}^{2H}}{\nn{x-y}^{2H}}, \quad x,y\in\R^n.
\]
\end{example}

\begin{example}\label{example:1}
In the case $\MM = \S^n$ and $\dd$ is the geodesic, we can take $\calH_x$ as the hemisphere of $\S^n$ centered as $x\in\S^n$,  $o\in\S^n$ a fixed point, $A_x := \calH_x\Delta  \calH_o, x\in\S^n$ and $\mu(\d x) = \pi \d x/w_n$ where $w_n$ is the surface area of $\S^n$. In this way, \eqref{eq:MDK} and \eqref{eq:limit K} follow.

The above choice is well-known when studying the spherical fractional Brownian motion. Note that the spherical fractional Brownian motion and the corresponding log-correlated Gaussian process are not rotation-invariant. For a rotation-invariant model, one can consider simply $A_x = \calH_x$. This time, we still have $\dd(x,y) = \mu(\calH_x\Delta \calH_y)$ (since $(\calH_x\Delta \calH_o)\Delta(\calH_y\Delta \calH_o) = \calH_x\Delta \calH_y$). But now we have $\mu(\calH_x) = \pi/2$ (compared with $\mu(A_x) = \dd (o,x)$) and hence the corresponding generalized Gaussian process has covariance kernel
\[
\Gamma\topp H(x,y) = \log\frac{2(\pi/2)^{2H}}{\dd^{2H}(x,y)}. 
\]
\end{example}
For other examples, see \citep[Section 2.8]{istas12manifold}.

The fact that $G^H(f)$ has  a stochastic-integral representation with $f$ from a suitable family of test functions yields immediately that the generalized Gaussian process $G$ can be extended to generalized infinitely-divisible processes \citep{samorodnitsky16stochastic} by replacing in the stochastic-integral representation the Gaussian random measure by infinitely-divisible ones (see \citep{bacry03log,rhodes10multidimensional} for a closely related example in a similar spirit). We do not pursue the full generality in this direction here. Instead, we explain that the extension to stable ones is almost free.  Last but not least, as the second main contribution we provide limit theorems for the introduced log-correlated bi-fractional Brownian motions. 

{\bf The paper is organized as follows.}  A unified treatment for general $(\MM,\dd)$ satisfying Assumption \ref{assump:1} and with $H\in(0,1/2]$ is provided in Section \ref{sec:general}. However, the presentation might be a little too abstract for the first time reading. To help illustrate, we shall focus on Example \ref{example:d=1}, the log-correlated bi-fractional Brownian motion on $\R_+ =[0,\infty)$ with covariance function \eqref{eq:Gamma d=1} and $H = 1/2$. For this case, in Section \ref{sec:d=1} we focus on stochastic-integral representations, and in Section \ref{sec:CLT} as the second main contribution we provide a central limit theorem based on a variation of a model by \citet{enriquez04simple} when studying the fractional Brownian motion. This is also the case closely related to a few very recent developments in the literature; see Remarks \ref{rem:NR} and \ref{rem:Karlin}. In Section \ref{sec:general}, after presenting the results in the general framework, we also provide a central limit theorem for $G$ based on another model similar to the one in \citet{fu20stable}.

\subsection{Related results in the literature}
We conclude the introduction with a few remarks.

\begin{remark}\label{rem:NR}
In an inspiring paper, \citet{neuman18fractional} considered to characterize the limit process of fractional Brownian motions, denoted by $\B^H$, as $H\downarrow 0$.
Therein, the authors investigated
\equh\label{eq:NR'}
X_H(t):=\frac1{\sqrt H}\pp{\B^H_t - \frac 1t\int_0^t\B_s^H\d s},
\eque
provided an argument on why a shift also makes sense in practice in finance applications, and then showed the limit of $X_H$ is a log-correlated Gaussian process. Later on, \citet{hager22multiplicative} showed that the random shift $t\inv\int_0^t\B_s^H\d s$ can be relaced by a larger family of processes and the limits remain to be  log-correlated (and also they considered $\R^d$-indexed random fields for $d\in\N$). Our process corresponds to a specific choice of the shift, essentially related to the decomposition of a fractional Brownian motion by a bi-fractional Brownian motion and a smooth Gaussian process due to \citet{lei09decomposition}. We provide more details in Remark \ref{rem:NR'}.
\end{remark}
\begin{remark}\label{rem:Karlin}
The limit process $G^{1/2}$ introduced here can already be read from recent developments by \citet{iksanov22small} on the Karlin model \citep{karlin67central,gnedin07notes}, and it is from their results we first had the guess of the expression of $G^{1/2}$. 
The Karlin model is an infinite urn model of which the law of the sampling over $\N$ decays as $j^{-1/{2H}}, H\in(0,1/2)$ as $j\to\infty$. The model was introduced by \citet{karlin67central}, who showed that the partial sum of so-called odd-occupancy process (the counting process of urns occupied by an odd number of balls), normalized by $n^H$, scales to a Gaussian random variable. The consideration of {\em odd-occupancy process} was inspired by  \citet{spitzer64principles} to modeling the on/off events for lightbulbs. It is known that randomized odd-occupancy processes scale to a fractional Brownian motion with Hurst index $H\in(0,1/2)$ \citep{durieu16infinite}, and this type of limit theorems has been extended to stable processes indexed general metric spaces \citep{durieu20infinite,fu20stable}. The odd-occupancy mechanism is reflected in the stochastic-integral representations (see \eqref{eq:fBm SI} below), and also in our representation for log-correlated bi-fractional Brownian motions.

The Karlin model with $H=0$ has been recently investigated by \citet{iksanov22functional} and \citet{iksanov22small}. Limit theorems have been established in this regime with new Gaussian processes in the limit in the aforementioned papers.
Their representation is not immediately recognizable to be related to ours and we provide some details in Appendix \ref{sec:Sasha}. It should be expected that from their model a counterpart of our Theorem \ref{thm:1} can be established (they examined $j$-occupancy processes instead of odd-occupancy processes). 
\end{remark}

\begin{remark}
 \citet{forde22riemann} considered the limit of re-scaled Riemann--Liouville process  and showed that the limit is a generalized Gaussian process with covariance function 
\[
\log{\frac{\sqrt t+\sqrt s}{\sqrt t-\sqrt s}}, \quad 0\le s<t<\infty.
\] 
Their limit theorem is in the same spirit as \citet{neuman18fractional}.

\citet{lifshits15bifractional} also considered a limit theorem for $G^{H,K}$ as $K\downarrow 0$, although a Lamperti's transformation is first applied to $G^{H,K}$ so that it is transformed to a stationary process before the limit is taken. A generalized process again arises in the limit.

Another example of a fractional Brownian motion with Hurst index $H = 0$ has been seen in the studies of random matrix theory \citep{fyodorov16fractional,lambert18mesoscopic}. The corresponding example is a centered self-similar Gaussian process with stationary increments. 
\end{remark}

\subsection*{Acknowledgements}
Y.W.~is grateful to Mikhail Lifshits for several stimulating discussions on an earlier version of the paper, which lead to an enhancement of the results and a major revision of the paper. Y.W.~thanks also Alexander Iksanov, Joseph Najnudel, and Yi Shen for several discussions. 
 Y.W.~was partially supported by Simons Foundation (MP-TSM-00002359), and a Taft Center Fellowship (2024--2025) from Taft Research Center at University of Cincinnati.

\section{A stochastic-integral representation}\label{sec:d=1}

In this and the next sections we consider $\MM = [0,\infty)$, $\dd(x,y) = |x-y|$. In most of this section we focus $H=1/2$ in the parameter of the bi-fractional Brownian motion except in Section \ref{sec:bifBm}. 

Our reference for stochastic integrals with respect  to stable random measures is \citet{samorodnitsky94stable}: if $(E,\calE,\mu)$ is a measure space, $M_\alpha$ is a symmetric $\alpha$-stable random measure on $(E,\calE)$ with control measure $\mu$,  then the {\em stochastic integral $\int f\d M_\alpha$ with control measure $\mu$} is defined as a symmetric $\alpha$-stable random variable with 
 \equh\label{eq:ch.f}
 \esp \exp\pp{\i \theta \int f\d M_\alpha} = \exp\pp{-|\theta|^\alpha\int |f|^\alpha\d\mu}, \mfa f\in L^\alpha(E,\mu), \theta\in\R.
 \eque
 We first focus on $\alpha=2$, and in this case $\int f\d M_2$ is a centered Gaussian random variable and $\cov(\int f\d M_2,\int g\d M_2) = 2\int fg\d \mu$ for $f,g\in L^2(E,\mu)$. (Note that the multiplicative constant $2$ is introduced to fit into the unified formula in \eqref{eq:ch.f}.)

Let $(\Omega,\calF,\proba)$ be the by-default probability space that the random measures are defined on. Set $E = \R_+\times\Omega'$ where $(\Omega',\calF',\proba')$ is another probability space on which a standard  Poisson process $N'$ is defined, and $M_{2}$ is a Gaussian random measure on $\R_+\times\Omega'$ with control measure $2 r\inv\d r\d\proba'$. Write
\[
\wt p_r(t) = \proba'(N'(rt)\odd) = \frac12\pp{1-e^{-2rt}}.
\]
We introduce
\equh\label{eq:G(f)}
G(f) :=\int_{\R_+\times\Omega'}\int_{\R_+}\pp{\inddd{N'(rt)\odd} - \wt p_r(t)}f(t)\d tM_{2}(\d r,\d\omega'),
\eque
and\equh\label{eq:F_0}\calF_0:=\ccbb{f\in L^\infty([0,\infty)): \int_1^\infty |f(s)|\log s\d s<\infty}.
\eque

\begin{lemma}
The collection $\{G(f)\}_{f\in \calF_0}$ is a family of centered Gaussian random variables with covariance function 
\[
\cov(G(f),G(g)) = \int_0^\infty\int_0^\infty f(s)g(t)\log\frac{s+t}{|s-t|}\d s\d t, \quad \mfa f,g\in \calF_0,
\]
that also satisfies
\equh\label{eq:linearity}
G(a_1f_1+a_2f_2) = a_1G(f_1)+a_2G(f_2) \mbox{ almost surely, for all } a_1,a_2\in\R, f_1,f_2\in\calF_0.
\eque
\end{lemma}

\begin{proof}
We claim
that
\begin{align}
\cov(G(f),G(f))& = 2\int_{\R_+\times\Omega'}\pp{\int_0^\infty\pp{\inddd{N'(rt)\odd}-\wt p_r(t)}f(t)\d t}^2\frac{2\d r}r\d\proba'
\nonumber\\
&=\int_{\R_+^2}f(s)f(t)\Gamma(s,t)\d s\d t.\label{eq:0}
\end{align}
The first equality follows from the property of stochastic integrals. 
Next we show the expression on the right-hand side of \eqref{eq:0} is finite.  We prove $\int_0^\infty \int_s^\infty |f(s)f(t)\log((s+t)/(t-s))|\d t\d s<\infty$. To see this, we observe
\[
\int_0^\infty \int_{s+1}^\infty |f(s)f(t)|\abs{\log{\frac{s+t}{t-s}}}\d t\d s\le \int_0^\infty |f(s)|\int_{s+1}^\infty |f(t)|\log(4t)\d t\d s<\infty,
\]
\[
\int_0^\infty \int_s^{s+1}\abs{f(s)f(t)\log\frac1{t-s}}\d t\d s\le \nn f_\infty \int_0^\infty |f(s)|\d s \int_0^1|\log t|\d t<\infty, 
\]
\[
\int_1^\infty \int_s^{s+1}\abs{f(s)f(t)}\log(s+t)\d t\d s\le \int_1^\infty \int_s^{s+1}|f(s)f(t)|\log (2t)\d t\d s<\infty,
\]
and
\begin{align*}
\int_0^1 \int_s^{s+1}\abs{f(s)f(t)\log(s+t)}\d t\d s &\le \nn f_\infty^2\int_0^1\int_s^{s+1}\abs{\log (s+t)}\d t\d s\\
& \le \nn f_\infty^2\int_0^1|\log s|+\log(1+s)\d s<\infty.
\end{align*}

It remains to prove the equality in \eqref{eq:0}, which follows from a careful application of Fubini's theorem. Indeed, for all $0\le s<t$,
\begin{multline}\label{eq:key1}
\esp'\pp{\pp{\inddd{N'(rs)\odd} - \wt p_s(r)}\pp{\inddd{N'(rt)\odd}-\wt p_t(r)}}\\
 = \frac14 \pp{(1-e^{-2rs})(1+e^{-2r(t-s)}) - (1-e^{-2rs})(1-e^{-2rt})} = \frac14\pp{e^{-2r(t-s)} - e^{-2r(t+s)}}.
\end{multline}
We then start by rewriting the expression after first equality in \eqref{eq:0} as
\begin{align*}
4\int_{\R_+}&\int_{\Omega'}\int_{\R_+^2}\pp{\inddd{N'(rs)\odd}-\wt p_r(s)}\pp{\inddd{N'(rt)\odd}-\wt p_r(t)}f(s)f(t)\d s\d t\d\proba'\frac{\d r}r\\
&=4\int_{\R_+}\int_{\R_+^2}\int_{\Omega'}\pp{\inddd{N'(rs)\odd}-\wt p_r(s)}\pp{\inddd{N'(rt)\odd}-\wt p_r(t)}f(s)f(t)\d\proba'\d s\d t\frac{\d r}r\\
& =\int_{\R_+}\int_{\R_+^2}\frac{e^{-2r|t-s|}-e^{-2r(t+s)}}r f(s)f(t)\d s\d t\d r.
\end{align*}
Here in the first step we applied Fubini's theorem to interchange the two inner integrals (in this step when verifying the condition in Fubini's theorem for each $r>0$, we bounded 
$\sabs{\inddd{N'(rs)\odd}-\wt p_r(s)}\sabs{\inddd{N'(rt)\odd}-\wt p_r(t)}\le 1$ and used the assumption $\nn f_1<\infty$), and in the second step we applied \eqref{eq:key1}. In the last expression above, we can interchange the two integrals again by Fubini's theorem, which, 
in combination of Frullani integral 
\equh\label{eq:log cov}
\int_0^\infty \frac{e^{-r(t-s)} - e^{-r(t+s)}}r\d r =  \log{\frac{t+s}{t-s}}, \quad 0\le s<t,
\eque
 leads to the equality in~\eqref{eq:0} (in this step we also need the fact that the right-hand side of \eqref{eq:0} is finite, which we have shown earlier).

Once we have shown that \eqref{eq:G(f)} is a well-defined stochastic integral, the linearity relation \eqref{eq:linearity} holds immediately by linearity of stochastic integrals. Then using $\cov(X,Y) = (\cov(X+Y, X+Y) - \cov(X,X) - \cov(Y,Y))/2$ we obtain the stated covariance formula.
\end{proof}

Next we introduce an approximation of $G(f)$. For each $\epsilon>0$, set
\equh\label{eq:M(f)}
G\topp\epsilon(f):=\int_{\R_+\times\Omega'}\int_{\R_+}\pp{\inddd{N'(rt)\odd} - \wt p_r(t)}\inddd{r\le \epsilon\inv}f(t)\d tM_{2}(\d r\d\omega'),
\eque
where we follow the same notations for stochastic integrals earlier. 
Note that by stochastic Fubini theorem \citep[Proposition 5.13.1]{peccati11wiener}, we have the representation
\equh\label{eq:G_epsilon approx}
G\topp\epsilon(f) = \int_0^\infty  \wt G\topp\epsilon_tf(t)\d t, \mbox{ almost surely,}
\eque
where 
\[
\wt G\topp\epsilon_t :=
\int_{\R_+\times\Omega'}\pp{\inddd{N'(rt)\odd} - \proba'(N'(rt)\odd)}\inddd{r\le\epsilon\inv}M_{2}(\d r\d\omega'), t\ge 0,
\]
is a point-wise defined stochastic process. The identity \eqref{eq:G_epsilon approx} is often referred  in the literature to as the point-wise representation of $G\topp\epsilon$ as a generalized function.

\begin{lemma}~
\begin{enumerate}[(i)]
\item
For each $\epsilon>0$, the process $\{\wt G_t\topp\epsilon\}_{t\ge0}$ is a centered Gaussian process with
\[
\cov\pp{\wt G\topp\epsilon_s,\wt G\topp\epsilon_t} = \int_0^{\epsilon\inv}\frac{e^{-2r|t-s|} - e^{-2r(t+s)}}r\d r. 
\]
\item For each $\epsilon>0$, we have
\[
\cov\pp{G\topp\epsilon(f),G\topp\epsilon(g)} = \int_0^\infty\int_0^\infty f(s)g(t)\int_0^{\epsilon\inv}\frac{e^{-2r|t-s|}-e^{-2r(t+s)}}r\d r\d s\d t, f,g\in\calF_0.
\]
In particular, $ \int_0^\infty \wt G_t\topp\epsilon f(t)\d t\weakto G(f)$ 
as $\epsilon\downarrow0$ for all $f\in\calF_0$. 
\end{enumerate}
\end{lemma}
\begin{proof}
For the first part, we have
\begin{align}
\cov \pp{\wt G\topp\epsilon_s,\wt G\topp\epsilon_t}  &= 4\int_{[0,\epsilon\inv]\times\Omega'}\pp{\inddd{N'(rs)\odd} - \wt p_s(r)}\pp{\inddd{N'(rt)\odd}-\wt p_t(r)}\frac{\d r}r\proba'(\d \omega')\nonumber\\
&= 4\int_0^{\epsilon\inv}\esp'\pp{\pp{\inddd{N'(rs)\odd} - \wt p_s(r)}\pp{\inddd{N'(rt)\odd}-\wt p_t(r)}}\frac{\d r}r,\label{eq:Fubini}
\end{align}
where in \eqref{eq:Fubini} we applied Fubini's theorem. In fact, by a straightforward calculation we have
\[
\esp\abs{\pp{\inddd{N'(rs)\odd} - \wt p_s(r)}\pp{\inddd{N'(rt)\odd}-\wt p_t(r)}}= \frac14\pp{1-e^{-4rs}}.
\]
(This function is not integrable with respect to $\d r/r$ at infinity, and hence we introduced the truncation $\inddd{r\le \epsilon\inv}$.)
The covariance function in the second part follows from a similar calculation. The convergence in fact can be read directly from \eqref{eq:M(f)} by letting $\epsilon\downarrow 0$ (compared with \eqref{eq:G(f)}).
\end{proof} 
\subsection{An extension to a generalized stable process}
The stochastic-integral representation \eqref{eq:G(f)} provides an immediate extension to a generalized infinitely-divisible process \citep{samorodnitsky16stochastic} by replacing the Gaussian random measure accordingly. We only describe an extension to stable ones for the sake of simplicity. For each $\alpha\in(0,2)$, we extend the definition of $G(f)$ in \eqref{eq:G(f)} to 
\[
G_\alpha (f) :=\int_{\R_+\times\Omega'}
\int_{\R_+}\pp{\inddd{N'(rt)\odd} - \wt p_r(t)}f(t)\d tM_{\alpha}(\d r,\d\omega'), 
\]
where $M_\alpha$ is a symmetric $\alpha$-stable random measure on $[0,\infty)\times\Omega'$ with control measure $2r\inv\d r\d\proba'$. 
So $G(f) = G_2(f)$ and we again have $G_\alpha(a_1f_1+a_2f_2) = a_1G_\alpha(f_1)+a_2G_\alpha(f_2)$ almost surely for $a_1,a_2\in\R, f_1,f_2$ from a suitably chosen class of functions. Recall that $M_\alpha$ is a symmetric $\alpha$-stable random measure and the corresponding stochastic integrals are random variables characterized by \eqref{eq:ch.f}. Formally, we have
\equh\label{eq:F alpha}
\esp \exp\pp{\i G_\alpha(f)} = \exp\pp{-2\int_0^{\infty}\esp'\abs{\int_0^\infty \pp{\inddd{N'(rt)\odd} - \wt p_r(t)}f(t)\d t}^\alpha\frac{\d r}r}.
\eque
Note that when $\alpha\in(0,2)$ it is not immediately clear what is the exact class of functions so that the right-hand side above is finite. Recall $\calF_0$ in \eqref{eq:F_0} and 
introduce
\equh\label{eq:f}
\calF_\delta :=\ccbb{f\in L^\infty([0,\infty)): \int_1^\infty s^\delta |f(s)|\d s<\infty}, \delta>0.
\eque
 In particular, $f\in\calF_\delta$ for any $\delta\ge 0$ implies that $\nn f_1<\infty$. 

\begin{lemma}
Assume one of the following:
\begin{enumerate}[(i)]
\item $\alpha\in(0,2], f\in\calF_\delta$ for some $\delta>0$, or
\item $\alpha\in(1,2], f\in \calF_0$.
\end{enumerate}
Then,  $G_\alpha(f)$ is a symmetric $\alpha$-stable random variable with characteristic function \eqref{eq:F alpha} (which is finite).
\end{lemma}
\begin{proof}
Again, by properties of stochastic integrals, it is equivalent to show that
\begin{multline}
\int_{\R_+\times\Omega'}
\abs{\int_{\R_+}\pp{\inddd{N'(rt)\odd} - \wt p_r(t)}f(t)\d t}^\alpha \frac{\d r}r\d\proba'\\
= \int_0^\infty \esp' \abs{\int_0^\infty\pp{\inddd{N'(rt)\odd} - \wt p_r(t)}f(t)\d t}^\alpha\frac{\d r}r<\infty.\label{eq:alpha finite}
\end{multline}
Without loss of generality, assuming $f\ge 0$. 
Introduce
\[
\Psi_f(r):= \int_0^\infty\int_s^\infty f(s) f(t)\pp{e^{-r(t-s)}-e^{-r(t+s)}}\d t\d s.
\]
We shall show
\equh\label{eq:Psi_f ineq}
\Psi_f(r)\le \begin{cases}
C(r^\delta\wedge r\inv),& \mbox{ if } f\in \calF_\delta, \delta>0,\\
C(|\log r|^{-2}\wedge r\inv), & \mbox { if } f\in \calF_0,
\end{cases}
\quad\mfa r>0,
\eque
where the constant $C$ depends on $f$ but not $r$.   
Then, by H\"older's inequality, we have
\begin{align*}
\esp \abs{\int_0^\infty\pp{\inddd{N'(rt)\odd} - \wt p_r(t)}f(t)\d t}^\alpha  & \le \pp{
\esp \abs{\int_0^\infty\pp{\inddd{N'(rt)\odd} - \wt p_r(t)}f(t)\d t}^2 }^{\alpha/2},
\\
&=  \pp{2\Psi_{f}(r)}^{\alpha/2},
\end{align*}
and hence \eqref{eq:alpha finite} becomes
$\int_0^\infty \Psi_f(r)^{\alpha/2}r\inv\d r<\infty$,
which now follows from \eqref{eq:Psi_f ineq}.

It remains to prove \eqref{eq:Psi_f ineq}. 
First, for all $r>0$ we have
\[
\Psi_f(r) 
\le \int_0^\infty f(s)\int_0^\infty f(s+t)e^{-rt}\d t\d s \le \nn f_\infty \frac 1r \int_0^\infty f(s)\d s  \le \nn f_\infty\nn f_1r^{-1}.
\]
Next, assuming $f\in\calF_\delta$ with $\delta>0$ first. 
By a change of variables write first
\[
\Psi_f(r)  = \int_0^\infty \int_0^\infty f(s)f(s+t)e^{-rt}(1-e^{-2rs})\d t \d s.
\]
 Next we prove for all $r>0$,  $\Psi_f(r)\le (2\nn f_1\int_0^\infty s^\delta f(s)\d s) r^\delta$. Indeed, we have
\begin{align*}
\int_0^{1/r}\int_0^\infty f(s)f(s+t)e^{-rt}(1-e^{-2rs})\d t\d s &\le 2r\int_0^{1/r}sf(s)\int_s^\infty f(t)\d t\d s\\
& \le 2r\nn f_1\int_0^{1/r}(1/r)^{1-\delta}s^\delta f(s)\d s \\
& \le 2r^\delta \nn f_1\int_0^\infty s^\delta f(s)\d s,
\end{align*}
and
\begin{align*}
\int_{1/r}^\infty\int_0^\infty f(s)f(s+t)e^{-rt}(1-e^{-2rs})\d t\d s &
\le \int_{1/r}^\infty f(s)\int_s^\infty f(t)\d t\d s\\
& \le \nn f_1 \int_{1/r}^\infty (1/r)^{-\delta} s^\delta f(s)\d s  \\
&\le r^\delta \nn f_1\int_0^\infty s^\delta f(s)\d s.
\end{align*}
We have proved \eqref{eq:Psi_f ineq} with $f\in\calF_\delta, \delta>0$.

Now assume $f\in \calF_0$. This time, fix $\gamma\in(0,1)$. We have, 
\begin{align*}
\int_0^{(1/r)^\gamma}\int_0^\infty f(s)f(s+t)e^{-rt}(1-e^{-2rs})\d t\d s &\le 2r\int_0^{(1/r)^\gamma}sf(s)\int_s^\infty f(t)\d t\d s\\
& \le 2r^{1-\gamma}\nn f_1^2,
\end{align*}
and, for $r<1/2$,
\begin{align*}
\int_{(1/r)^\gamma}^\infty f(s)\int_s^\infty f(t)\d t\d s & \le \int_{(1/r)^\gamma}^\infty \frac{f(s)}{\log s}\int_s^\infty f(t)\log t\d t\d s\\
& \le \frac1{\gamma^2\log^2 r}\int_{1/r}^\infty  f(s)\log s\d s \int_s^\infty  f(t)\log t\d t \le C\frac1{\log^2 r}.
\end{align*}
So for $r<1/2$ we have $\Psi_f(r)<C/\log ^2 r$. By possibly increasing $C$ we have proved  \eqref{eq:Psi_f ineq} when $f\in\calF_0$. 
\end{proof}

\subsection{A representation for the bi-fractional Brownian motion with \texorpdfstring{$H\in(0,1/2)$}{}} 
\label{sec:bifBm}
In the previous discussions in this section, we essentially worked with $G^{1/2,2H}$, of which the stochastic-integral representations have been known. Here, we provide a new stochastic-integral representation for $G^{H,K}$ with $H\in(0,1/2)$ (to be compared with the case $H=1/2$, which we recall in \eqref{eq:G SI} below). 
Consider
\[
\wt G^{H,K}(t) = \int_{\R_+\times\Omega'}\pp{\inddd{N'(\calS_{2H,r}'t)\odd} - \wt p_{2H,r}(t)}\wt M_{2,K}(\d r\d \omega'),
\]
where this time $\wt M_{2,K}$ is on $\R_+\times\Omega'$ with control measure $C_K r^{-K-1}\d r\d\proba'$ where $\proba'$ is a probability measure on $\Omega'$, under which $N'$ is a standard Poisson process and $\calS_{2H,r}$ is such that $\esp' e^{-\theta\calS'_{2H,r}} = e^{r\theta^{2H}}$ for all $\theta\ge 0$, $N'$ and $\calS_{2H,r}'$ are independent, and 
\[
\wt p_{2H,r}(t) := \proba'\pp{N'(\calS_{2H,r}'t)\odd} = \esp'\pp{\frac12\pp{1-e^{-2\calS'_{2H,r}t}}}= \frac12\pp{1-e^{-(2t)^{2H}r}}.
\]
\begin{lemma}
Under the notations above, for all $H\in(0,1/2]$ and $K\in(0,1)$ we have
\[
\cov\pp{\wt G^{H,K}(s),\wt G^{H,K}(t)} = \frac{2^{(2H-1)K}}4 \pp{(s^{2H}+t^{2H})^K - (t-s)^{2H K}}, 0\le s<t<\infty. 
\]
\end{lemma}

\begin{proof}Assume $s<t$. We first compute (with respect to $\proba'$)
\[
\cov'\pp{\inddd{N'(\calS_{2H,r}'s)\odd},\inddd{N'(\calS_{2H,r}'t)\odd}} = \frac14 \pp{e^{-2^{2H}(t-s)^{2H} r} - e^{-2^{2H}(s^{2H}+t^{2H})r}}.
\]
Then, 
\begin{align*}
\cov\pp{\wt G^{H,K}(s),\wt G^{H,K}(t)} & =  C_K\int_0^\infty r^{-K-1}\d r \cov'\pp{\inddd{N'(\calS_{2H,r}'s)\odd},\inddd{N'(\calS_{2H,r}'t)\odd}}\\
& = \frac{C_K}4\int_0^\infty \pp{e^{-2^{2H}(t-s)^{2H}r} - e^{-2^{2H}(s^{2H}+t^{2H})r}}r^{-K-1}\d r\\
& = \frac{C_K}4\frac{\Gamma(1-K)}K 2^{2HK}\pp{\pp{s^{2H}+t^{2H}}^K - (t-s)^{2HK}}.
\end{align*}
\end{proof}
Note that when $H = 1/2$, this representation was known \citep{durieu20infinite,fu20stable}. A new idea here for $H\in(0,1/2)$ is the introduction of a stable subordinator $\calS_{2H,r}'$ in the representation, and the same idea shall be applied in Section \ref{sec:general}. 
\begin{remark}
In fact, one can have accordingly stochastic representations \citep{fu20stable} for bi-fractional Brownian motions indexed by metric spaces. Furthermore, replacing Gaussian random measures by symmetric $\alpha$-stable ones, we obtain immediately extensions of bi-fractional Brownian motions, which we may name {\em bi-fractional stable motions}. We omit the details. 
\end{remark}
\begin{remark}\label{rem:NR'}

Our result can be interpreted as a limit theorem for $\B^H$ as $H\downarrow 0$ with a different shift from the one considered by \citet{neuman18fractional} in \eqref{eq:NR'}. Our choice of the shift is in fact related to the decomposition of the fractional Brownian motion by a bi-fractional Brownian motion due to \citet{lei09decomposition}.
Namely,  $\B^H$ satisfies
\equh\label{eq:lei09}
\pp{\B^H_t}_{t\in[0,\infty)}\eqd\pp{G^{1/2,2H}_t + W^H_t}_{t\in[0,\infty)},
\eque
where $G^{1/2,2H}$ is a bi-fractional Brownian motion introduced earlier and $W^H$ satisfies
\[
 \cov\pp{W^H_s,W^H_t} = \frac12\pp{s^{2H} + t^{2H} - (s+t)^{2H}},
\]
and the two processes are independent.
Moreover, the decomposition \eqref{eq:lei09} can hold in the almost sure sense. Let $(\Omega,\calF,\proba)$ be the by-default probability space that the random measures are defined on. Set $E = \R_+\times\Omega'$ where $(\Omega',\calF',\proba')$ is another probability space on which a standard  Poisson process $N'$ is defined, and $\wt M_{2,H}$ as a Gaussian random measure (defined on $(\Omega,\calF,\proba)$) on $E = \R_+\times\Omega'$ with control measure $C_{2H}r^{-2H-1}\d r\d\proba'$ with
$C_\beta = \beta 2^{-\beta}/\Gamma(1-\beta), \beta\in(0,1)$.
Then, \eqref{eq:lei09} holds with the following:
\begin{align}\label{eq:fBm SI}
\B_t^H&:=\int_{\R_+\times\Omega'}\inddd{N'(rt)\odd}\wt M_{2,2H}(\d r\d \omega'),\\
 G^{1/2,2H}_t&:= \int_{\R_+\times\Omega'}\pp{\inddd{N'(rt)\odd} - \wt p_r(t)}\wt M_{2,2H}(\d r\d \omega'),\label{eq:G SI}\\
W^H_t&:= \int_{\R_+\times\Omega'} \wt p_r(t)\wt M_{2,2H}(\d r\d \omega'),\nonumber
\end{align}
with $\wt p_r(t) = \proba'(N'(rt)\odd) = \pp{1-e^{-2rt}}/2$  \citep{durieu20infinite,fu20stable}.

Now, one can write 
\[
G^{1/2,2H}_t = \B^H_t-W_t^H.
\]
That is, $G^{1/2,2H}$ is $\B^H$ with a random shift $W^H$. The convergence of $H^{-1/2}G^{1/2,2H}\weakto G^{1/2}$ is now a special case of \eqref{eq:limit K}. 
\end{remark}
\section{A central limit theorem}\label{sec:CLT}

Another feature of the introduced process $G \equiv G^{1/2}$ is that it arises from a simple aggregated model. The aggregated model here is a variation of the one proposed by \citet{enriquez04simple} when investigating approximations of a fractional Brownian motion. Connections between aggregated models and fractional Gaussian processes and random fields with long-range dependence \citep{pipiras16long,samorodnitsky16stochastic} are well-known \citep{bierme10selfsimilar,breton11functional,kaj08convergence,mikosch07scaling} (in particular, \citep{bierme10selfsimilar,breton11functional} established limit theorems for generalized processes).

We follow the notations in \citet{shen23aggregated}, where the Enriquez model was revisited and extended. The model is an aggregated model of an increasing number, denoted by $m_n$, of i.i.d.~layers indexed by $i\in\N$, and each layer consists of a sequence of stationary random variables denoted by $(\calX_{i,j})_{j\in\N}$ that we now introduce. First, let $\calX$ be a standard Gaussian random variable, $q$ a uniform random variable taking values from $(0,1)$, and $\{\eta_{j}\topp{q}\}_{j\in\N}$  conditionally i.i.d.~Bernoulli random variables with parameter $q$. Assume that $\calX$ is independent from $q$ and $\{\eta_j\topp q\}_{j\in\N}$. 
Let $(\calX_i,q_i, \{\eta_{i,j}\topp{q_i}\}_{j\in\N})_{i\in\N}$ denote i.i.d.~copies of $(\calX,q,\{\eta_j\topp q\}_{j\in\N})$, and set
\[
\tau_{i,j}\topp{q_i} :=\summ k1j \eta_{i,k}\topp {q_i}, j\in\N.
\]
Then,  the $j$-th element of $i$-th layer is defined as  
\[
\calX_{i,j}:=\frac{\calX_i}{q^{1/2}}\pp{(-1)^{\tau_{i,j}\topp{q_i}+1}\eta_{i,j}\topp{q_i} - \esp \pp{(-1)^{\tau_{i,j}\topp{q_i}+1}\eta_{i,j}\topp{q_i}\mmid q_i}}.
\]
Now, the partial-sum process of the aggregated model is represented as
\[
\what S_{n,j} := \frac1{\sqrt{m_n}}\summ i1{m_n}\summ k1j \calX_{i,k}= \frac1{\sqrt{m_n}}\summ i1{m_n}\frac{\calX_i}{q_i^{1/2}}\pp{\inddd{\tau_{i,j}\topp{q_i}\odd} - \proba\pp{\tau_{i,j}\topp {q_i}\odd\mmid q_i}}, j\in\N. 
\]
In order to see that the proposed model exhibits log-correlated dependence structure, a quick calculation yields
\equh\label{eq:cov conv}
\limn \cov\pp{\what S_n(s),\what S_n(t)} = \frac14\log{\frac{s+t}{s-t}}, 0<s<t.
\eque
Indeed,  \begin{align*}
\cov\pp{\what S_n(s),\what S_n(t)} 
& = \int_0^{1}\frac 1r\Bigg[\pp{\frac 12\pp{1-\pp{1-2r}^{\floor{ns}}}\frac12\pp{1+(1-2r)^{\floor{n(t-s)}}}}\\
& \quad\quad - \frac14\pp{1-(1-2r)^{\floor{ns}}}\frac12\pp{1-(1-2r)^{\floor{nt}}} \Bigg]\d r\\
& \to \frac14 \int_0^\infty \frac1r \pp{e^{-2(t-s)r} - e^{-2(t+s)r}}\d r = \frac14\log{\frac{t+s}{t-s}}.
\end{align*}
(Note that for binomial random variable $\tau_n\topp p$ with parameter $(n,p)$, $\proba(\tau_n\topp p\odd) = (1/2)(1-(1-2p)^n)$ and $\proba(\tau_n\topp p{\rm~is~even}) = (1/2)(1+(1-2p)^n)$.)
 The step of convergence `$\to$' would require a little extra work. This is a standard analysis and details are omitted here;  we do not make use of \eqref{eq:cov conv} in the sequel. 

The convergence in \eqref{eq:cov conv} suggests that there is no convergence of point-wise defined stochastic processes, and instead one should search for convergence in the sense of generalized processes. For this purpose, recall $\calF_\delta$ in \eqref{eq:f} and introduce
\equh\label{eq:G_n(f)}
G_n(f):=2 \sif j1\what S_{n,j}\int_{(j-1)/n}^{j/n}f(x)\d x, \quad f\in \calF_\delta,
\eque
with $\delta>0$.
(Note that we do not assume $f$ to be continuous for $f\in\calF_\delta$. If continuity is assumed, then one can replace $\int_{(j-1)/n}^{j/n}f(x)\d x$ by $f(j/n)/n$ and establish the same limit theorem below.)

The second main result is  the following limit theorem. 
\begin{theorem}
\label{thm:1}
Assume $\limn m_n/n = \infty$ and $\delta>0$. Then, for all $f\in \calF_\delta$, 
\[
\limn\esp \exp\pp{\i G_n(f)} = \esp \exp\pp{\i G(f)} = \exp\pp{-\frac12\cov(G(f),G(f))}.
\]
In particular, $\{G_n(f)\}_{f\in\calF_\delta}\fddto \{G(f)\}_{f\in\calF_\delta}$ as $n\to\infty$.
\end{theorem}
\begin{remark}
We impose the  condition in $\calF_\delta$ as it is convenient for establishing our limit theorems.  In particular, $\calF_\delta$ includes all smooth functions that decay rapidly at infinity. 
L\'evy's continuity theorem yields immediately the convergence of $G_n$ in the space of tempered distributions. For the generalized process $G$ indexed by $\MM$, L\'evy's continuity theorem applies as long as the space of functions on $\MM$ is a nuclear space \citep{fernique68generalisations,meyer65theoreme}.  We do not pursue to strengthen the convergence in the space of distributions here. 
\end{remark}
\begin{remark}
The process $\what S_{n,j}$ can be thought of the $H=0$ counterpart of the model considered in \citet{enriquez04simple} and \citet{shen23aggregated}. Therein, the density of $q$ is asymptotically equivalent to $Cr^{-2H}$ as $q\downarrow 0$ for some constant $C>0$, and the normalization $1/\sqrt{m_n}$ in $\what S_{n,j}$ is replaced by $1/(n^H\sqrt{m_n})$, and it was shown that $\{\what S_{n,\floor{nt}}\}_{t\in[0,1]}\fddto c\{G^H_t\}_{t\in[0,1]}$ for some explicit constant $c$. (Therein, there is also no centering in $\what S_{n,j}$ and so the limit is actually $\B^H$ instead of $G^H$.)

\end{remark}
\subsection{Proof of  Theorem \ref{thm:1}} Since all $G_n(f)$ and $G(f)$ are linear in $f$, it suffices to show
\equh\label{eq:chf conv}
\limn\esp \exp\pp{\i G_n(f)} = \esp \exp\pp{\i G(f)} = \exp\pp{-\frac12 \cov( G(f),G(f))},
\eque
for all $f\in\calF_\delta,\delta>0$.
We start by evaluating $\esp e^{\i G_n(f)}$. We first introduce some notations. Let $(\eta_j\topp r)_{j\in\N}$ denote i.i.d.~Bernoulli random variables with parameter $r\in(0,1)$, and set $\tau_j\topp r := \summ k1j \eta_k\topp r$. 
Introduce  
\[
p_j(r) := \proba\pp{\tau_j\topp r\odd}, r\in(0,1). 
\]
Note that here $r$ is a fixed real number. We also let $q$ denote a uniform random variable, and accordingly we write $p_j(q) = \proba(\tau_j\topp q\odd\mid q)$. 
Write
\[
 \psi_{n,r}(f):=2\sif j1f_{n,j}\pp{\inddd{\tau_j\topp r\odd} - p_j(r)} \qmwith f_{n,j} := \int_{(j-1)/n}^{j/n}f(x)\d x.
\]
Let $\calX$ denote a standard Gaussian random variable. Since $G_n(f)$ is the sum of $m_n$ i.i.d.~random variables (recall \eqref{eq:G_n(f)}), we have
\begin{align*}
\esp e^{\i G_n(f)} & = \pp{\esp \exp\pp{\frac {2\i}{ \sqrt {m_n} }\sif j1f_{n,j}\frac{\calX}{q^{1/2}}\pp{\inddd{\tau_j\topp q\odd} -p_j(q)}}}^{m_n}\\
& =\pp{\int_0^1 \esp \exp\pp{-\frac1{2m_nr}\psi_{n,r}(f)^2}\d r}^{m_n}.
\end{align*}

The goal is to write the above as $(1+v_n(f)/m_n)^{m_n}$ with $v_n(f)\to -\cov(G(f),G(f))/2$ as $n\to\infty$. For this purpose, we write
\[
\int_0^1 \esp \exp\pp{-\frac1{2m_nr}\psi_{n,r}(f)^2}\d r = 1 + \int_0^1 \pp{\esp \exp\pp{-\frac1{2m_nr}\psi_{n,r}(f)^2}-1}\d r.
\]
Fix $\epsilon>0$. In order to prove \eqref{eq:chf conv} we shall prove
\begin{align}
\lim_{\epsilon\downarrow0}\limn  m_n\int_{\epsilon/n}^{\epsilon\inv/n}  \pp{\esp \exp\pp{-\frac1{2m_nr}\psi_{n,r}(f)^2}-1}\d r &= -\frac12\cov(G(f),G(f)),\label{eq:limit eps1}\\
\lim_{\epsilon\downarrow0}\limsupn 
m_n\int_{[0,1]\setminus[\epsilon/n,\epsilon\inv/n]}  \pp{\esp \exp\pp{-\frac1{2m_nr}\psi_{n,r}(f)^2}-1}\d r &= 0.\label{eq:limit eps2}
\end{align} 
\medskip

\noindent (i) We start by proving \eqref{eq:limit eps1}.
First,
\begin{multline}\label{eq:limit eps1'}
\int_{\epsilon/n}^{\epsilon\inv/n}  \pp{\esp \exp\pp{-\frac1{2m_nr}\psi_{n,r}(f)^2}-1}\d r\\
 = \frac 1n\int_{\epsilon}^{\epsilon\inv} \pp{\esp \exp\pp{-\frac n{2m_nr}\psi_{n,r/n}(f)^2}-1}\d r \sim - \frac 1{2m_n}\int_\epsilon^{\epsilon\inv}\frac1r\esp \psi_{n,r/n}(f)^2{\d r}.
\end{multline}
In the last step we used the assumption that $m_n\gg n$ and that 
\[
\limsupn\sup_{r\in[\epsilon,\epsilon\inv]}\frac{\psi_{n,r/n}(f)^2}{r}<\infty.
\] 
This is because for $f\in\calF_\delta$ we have for all $r>0$,
\[
|\psi_{n,r}(f)| \le 2\sif j1f_{n,j}\le 2\nn f_1<\infty.
\]

For each $r>0$ fixed, we shall establish the following.
\begin{lemma}\label{lem:var conv}
For all $r>0,\delta>0$ and $f\in\calF_\delta$, 
\begin{align*}
\limn \esp\psi_{n,r/n}(f)^2 &= 4\esp\pp{\int_0^\infty f(s)\pp{\inddd{N(rs)\odd} - \wt p_r(s)}\d s}^2\\
& = \int_0^\infty\int_0^\infty f(s)f(t)\log\frac{s+t}{|t-s|}\d s\d t.\nonumber
\end{align*}
\end{lemma}
\begin{proof}
Indeed, 
\[
\esp\psi_{n,r/n}(f)^2 = \esp\pp{2\sif j1f_{n,j}\pp{\inddd{\tau_j\topp {r/n}\odd} - p_j(r/n)}}^2.
\]
Fix a parameter $K\in\N$, and  write
\[
\psi_{n,r,K}(f) := 2\sum_{j=1}^{Kn}
f_{n,j}\pp{\inddd{\tau_j\topp r\odd} - p_j(r)},
\]
and $\wt\psi_{n,r,K}(f):=\psi_{n,r}(f)-\psi_{n,r,K}(f)$. We first show that
\equh\label{eq:variance conv}
\limn \esp\psi_{n,r/n,K}(f)^2 = 4\esp\pp{\int_0^Kf(t)\pp{\inddd{N_r(t)\odd} - \wt p_r(t)}\d t}^2.
\eque
It suffices to show
\equh\label{eq:CMT}
\summ j1{Kn}f_{n,j}\inddd{\tau_j\topp{r/n}\odd}\weakto \int_0^Kf(x)\inddd{N_r(x)\odd}\d x.
\eque
Then, since the random variables on the left-hand side are all bounded by $K\nn f_\infty$, we have the convergence of all moments, and in particular \eqref{eq:variance conv} holds.

The convergence \eqref{eq:CMT} essentially follows from the well-known Poisson convergence of triangular arrays of i.i.d.~Bernoulli random variables 
\equh\label{eq:PP convergence}
\summ j1{Kn} \eta_j\topp{r/n}\delta_{j/n}\weakto \PPP([0,K],r\d x),
\eque
where the right-hand side denotes the Poisson point process on $[0,K]$ with intensity measure $r\d x$, and a continuous-mapping argument. Recall that $r>0$ is fixed in \eqref{eq:CMT}. 

First, we re-state the point-process convergence in \eqref{eq:PP convergence} as the convergence of a sequence of random points.  Set $\what j_{n,1}:=\min\{j\in\N:\eta_j\topp{r/n} = 1\}$, 
\[
\what j_{n,k+1} :=\min\ccbb{j>\what j_{n,k}: \eta_j\topp{r/n} = 1}, k\in\N,
\]
and set 
\[
N_{K,n}:=\max\ccbb{k\in\N:\what j_{n,k}\le Kn}.
\]
So the point process on the left-hand side of \eqref{eq:PP convergence} is now $\summ k1{N_{K,n}}\delta_{\what j_{n,k}/n}$. Set also $\{\Gamma_j\}_{j\in\N}$ be the sequence of consecutive arrival times of the Poisson process $\{N(rt)\}_{t\ge 0}$, and set $N_K:=\max\{j\in\N_0:\Gamma _j\le K\}$. We then have $N_{K,n}\weakto N_K$ and
\equh\label{eq:seq conv}
\pp{\what j_{n,1},\dots,\what j_{n,N_{K,n}},0,\dots} \weakto \pp{\Gamma_1,\dots,\Gamma_{N_K},0,\dots}
\eque
as weak convergence of random elements in $\{(a_j)_{j\in\N}: a_j = 0 \mbox{ for $j$ large enough}\}$. 

The left-hand side of \eqref{eq:CMT} can be re-written as
\[
F_{K,n}:=\sum_{k=1}^{\sfloor{N_{K,n}/2}}\int^{\what j_{n,2k}/n}_{\what j_{n,2k-1}/n}f(t)\d t + \inddd{N_{K,n}\odd}\int^K_{\what j_{n,N_{K,n}}/n}f(t)\d t.
\]
To show \eqref{eq:CMT} now it suffices to show
\equh\label{eq:Skorokhod}
F_{K,n}\weakto F_K:= \sum_{k=1}^{N_K/2}\int^{\Gamma_{2k}}_{\Gamma_{2k-1}}f(t)\d t + \inddd{N_K\odd}\int_{\Gamma_{N_K}}^Kf(t)\d t.
\eque
The above follows from a continuous mapping theorem applied to \eqref{eq:seq conv}. One way to make this precise is to, by Skorokhod representation theorem, assume that the convergence in \eqref{eq:seq conv} is further in the almost sure sense, and then examine \eqref{eq:Skorokhod} for every $\omega$ such that the sequence convergences. Fix such an $\omega$, and notice that for $n$ large enough $N_{K,n}(\omega) = N_K(\omega)$, and then both $F_{K,n}$ and $F_K$ are expressed the same continuous function of $N_K(\omega)$  random variables. It is now clear that \eqref{eq:Skorokhod} converges almost surely.
We have thus proved \eqref{eq:CMT}, and hence \eqref{eq:variance conv}. 

To conclude the proof, it remains to show that
\[
\lim_{K\to\infty}\limsupn \esp\wt\psi_{n,r/n,K}(f)^2 = 0.
\]
For this purpose, we could simply bound 
\[
\esp\wt\psi_{n,r/n,K}(f)^2\le \pp{2\sum_{j=Kn+1}^\infty f_{n,j}}^2 \le 4\pp{\int_K^\infty |f(x)|\d x}^2,
\]
and the desired result follows from the fact that $\nn f_1<\infty$.
\end{proof}
Following Lemma \ref{lem:var conv}, we have
\[
\limn\int_{\epsilon/n}^{\epsilon\inv/n}\frac{\esp \psi_{n,r/n}(f)^2} r\d r = \int_0^\infty\int_0^\infty f(s)f(t)\int_\epsilon^{\epsilon\inv}\frac{e^{-2r(t-s)}-e^{-2r(t+s)}}r\d r\d s\d t,
\]
which as $\epsilon\downarrow 0$ converges to $\cov(G(f),G(f))$ by \eqref{eq:log cov}. Combining with \eqref{eq:limit eps1'}, we have thus proved \eqref{eq:limit eps1}. \medskip

\noindent (ii) Now we prove \eqref{eq:limit eps2}. applying the simple inequality
\[
1-\esp \exp\pp{-\frac1{2m_nr}\psi_{n,r}(f)^2}\le
 \frac 1{2m_nr}\esp\psi_{n,r}(f)^2,
\]
it suffices to show
\begin{align}\label{eq:r small}
\lim_{\epsilon\downarrow 0} \limsupn & \int_0^{\epsilon}\frac1{r}\esp \psi_{n,r/n}(f)^2\d r  = 0,\\
\lim_{\epsilon\downarrow 0}\limsupn &\int_{\epsilon\inv}^n\frac1{r}\esp \psi_{n,r/n}(f)^2\d r = 0.\label{eq:r large}
\end{align}
We first show \eqref{eq:r large}.
Note that, for $j\le j'$,
\begin{align*}
\proba & \pp{\tau_j\topp r\odd,\tau_{j'}\topp r\odd} - \proba\pp{\tau_j\topp r\odd}\proba\pp{\tau_{j'}\topp r\odd}\\
& = \frac14\pp{\pp{1-(1-2r)^j}\pp{1+(1-2r)^{j'-j}} - \pp{1-(1-2r)^j}\pp{1-(1-2r)^{j'}}}\\
& = \frac14\pp{1-(1-2r)^j}\pp{(1-2r)^{j'-j}+(1-2r)^{j'}} = \frac14\pp{1-(1-2r)^{2j}}(1-2r)^{j'-j}.
\end{align*}
Therefore, we have
\begin{align}
\esp \psi_{n,r/n}(f)^2 & \le  8\sif j1 \sum_{j'=j}^\infty \abs{f_{n,j}f_{n,j'}}
 \frac14\pp{1-\pp{1-\frac{2r}n}^{2j}}\pp{1-\frac{2r}n}^{j'-j}\nonumber\\
 & = 8\sif j1 \sum_{k=0}^\infty \abs{f_{n,j}f_{n,j+k}}
 \frac14\pp{1-\pp{1-\frac{2r}n}^{2j}}\pp{1-\frac{2r}n}^k\label{eq:var formula}\\
 & \le 2\sif j1 |f_{n,j}|\sif k0\frac{\nn f_\infty}n\pp{1-\frac{2r}n}^k = \frac{\nn f_\infty \nn f_1}{r}.\nonumber
\end{align}
Thus,
\[
\int_{\epsilon\inv}^n \frac{\esp \psi_{n,r/n}(f)^2}r\d r  \le C\int_{\epsilon\inv}^\infty \frac 1{r^2}\d r  = C\epsilon. 
\]
The desired \eqref{eq:r large} now follows. 

We next show \eqref{eq:r small}. 
Break the expression \eqref{eq:var formula} into two parts depending on $j/n\le 1/r$ and $j/n>1/r$. First,
\begin{align*}
 &   8\sum_{j=1}^{\floor{n/r}} \abs{f_{n,j}}\sif k0 \abs{f_{n,j+k}} \frac14\pp{1-\pp{1-\frac{2r}n}^{2j}}\pp{1-\frac{2r}n}^k \\
&\le 8r\sum_{j=1}^{\floor{n/r}} \frac jn\abs{f_{n,j}}\nn f_1 \le  
 8r^\delta\sum_{j=1}^{\floor{n/r}} \pp{\frac jn}^\delta\abs{f_{n,j}}\nn f_1
 \le C  r^\delta  \int_0^\infty s^\delta |f(s)|\d s\nn f_1,
\end{align*}
where in the first step we used  $1-x^p = (1-x)(1+x+\cdots+x^{p-1})\le p(1-x)$ for $x\in[-1,1], p\in\N$. Also,
\begin{align*}
 &  8\sum_{j=\floor{n/r}+1}^\infty \abs{f_{n,j}}\sif k0 \abs{f_{n,j+k}} \frac14\pp{1-\pp{1-\frac{2r}n}^{2j}}\pp{1-\frac{2r}n}^k \\
&\le 8\sum_{j=\floor{n/r}+1}^\infty \abs{f_{n,j}}\nn f_1 \le 8r^\delta
\sum_{j=\floor{n/r}+1}^\infty \pp{\frac jn}^\delta \abs{f_{n,j}} \nn f_1\le C r^\delta \int_0^\infty s^\delta |f(s)|\d s\nn f_1.
\end{align*}
That is, we have shown that  $\esp \psi_{n,r/n}(f)^2\le Cr^\delta$, whence
\[
\int_0^{\epsilon}\esp\psi_{n,r/n}(f)^2\frac{\d r}r \le C \epsilon^\delta,
\]
and hence  \eqref{eq:r large} follows. This completes the proof of Theorem \ref{thm:1}.
\section{Log-correlated bi-fractional Brownian motions  indexed by metric spaces}\label{sec:general}
We introduce a family of log-correlated generalized Gaussian processes indexed by metric spaces $(\MM,\dd)$ under the Assumption \ref{assump:1} and their stable extensions in a unified framework. We also provide a central limit theorem. 

\subsection{A stochastic-integral representation} Let $(E,\calE,\mu)$ and $\{A_x\}_{x\in\MM}\subset\calE$ be as in Assumption \ref{assump:1}.  Let $\M_p(E)$ denote the space of Radon point measures on $E$, and $\calM_p(E)$ its Borel $\sigma$-algebra (see e.g.~\citep{resnick87extreme}). For notational convenience, we write 
\[
\beta = 2H\in(0,1]
\] in this section.
 For each $r>0$, let $\proba_{\beta,r}'\equiv \proba_{\beta,r,\mu}'$ denote the probability measure on $(\M_p(E),\calM_p(E))$ induced by the Poisson point process on $(E,\calE)$ with random intensity measure $\calS_{\beta,r} \cdot \mu$, where $\calS_{\beta,r}$ is a totally skewed $\beta$-stable random variable with 
\[
\esp \exp(-\theta \calS_{\beta,r}) = \exp\pp{-r\theta^\beta}, \theta>0.
\]
(That is, $\calS_{\beta,r}$ is a $\beta$-stable subordinator evaluated at time $r$.)
Note that the special case $\calS_{1,r} = r$ (and this section might be easier to read first assuming $\beta=1$). 
In particular, letting $m$ denote an element in $\M_p(E)$, 
\begin{align*}
\proba'_{\beta,r}(m(A) = k)& \equiv \proba'_{\beta,r}\pp{\ccbb{m\in\mathbb M_p(E):m(A) = k}} \\&=  \esp\pp{\frac{(\calS_{\beta,r}\mu(A))^k}{k!}e^{-\calS_{\beta,r}\mu(A)}}, k\in\N_0, 
\end{align*}
for all $A\in\calE, \mu(A)<\infty$. 
Set
\[
\wt p_{\beta,r}(x):= \proba'_{\beta,r}\pp{m(A)\odd} = \esp\pp{\frac12\pp{1-e^{-2\calS_{\beta,r}\mu(A_x)}}} = \frac12\pp{1-e^{-(2\mu(A_x))^\beta r}}, x\in\MM.
\]
Let $\lambda$ be a Radon measure on $(\MM,\dd)$.  Introduce, for $\alpha\in(0,2]$, 
\[
G_{\alpha,\beta}(f) :=\int_{\R_+\times\M_p(E)}\int_{x\in\MM}\pp{\inddd{m(A_x)\odd}-\wt p_{\beta,r}(x)}f(x)\lambda(\d x)M_{\alpha,\beta}(\d r\d m),
\]
where  $M_{\alpha,\beta}$ is a symmetric $\alpha$-stable random measure on $\R_+\times\M_p(E)$ with control measure $2r\inv\d r\d \proba_{\beta,r}')$, 
and $f\in\calF_{\delta,\delta',\beta}(\MM)$ defined as follows.
\begin{assumption}\label{assump:2}
With $\delta,\delta'>0, \beta\in(0,1]$, we let $\calF_{\delta,\delta',\beta}(\MM)$ denote the class of all bounded and measurable functions $f:\MM\to\R$ (measurable with respect to the Borel $\sigma$-algebra induced by $\dd$) satisfying the following:
\begin{enumerate}[(i)]
\item $\nn f_1:=\int_\MM |f(x)|\lambda(\d x)<\infty$,
\item $\nn f_{1,\delta}:=\int_\MM \mu(A_x)^{\delta} |f(x)|\lambda(\d x)<\infty$,
\item $\int_\MM\int_\MM e^{-\dd^\beta(x,y)r}|f(x)f(y)|\lambda (\d x)\lambda(\d y) \le Cr^{-\delta'}$ for all $r>0$.
\end{enumerate}
\end{assumption}
It is useful to recall that when $\mu(A_o) = 0$ for some $o\in\MM$ (usually playing the role of the origin), we have $\mu(A_x) = \dd(o,x)$. In particular, the above assumption can be viewed as a generalization of the assumption on $f$ in \eqref{eq:f}. 

The following is the main result for log-correlated bi-fractional Brownian motions and their stable extensions indexed by metric spaces, where $\Gamma\topp\beta(x,y)$ is  the same as in Theorem \ref{thm:0}. 
\begin{theorem}\label{thm:Kahane}For each $\alpha\in(0,2], \delta,\delta'>0$, $\{G_{\alpha,\beta}(f)\}_{f\in \calF_{\delta,\delta',\beta}(\MM)}$ is a family of symmetric $\alpha$-stable random variables with
\equh\label{eq:G M}
\esp \exp\pp{\i G_{\alpha,\beta}(f)} = \exp\pp{-2\int_0^\infty\esp'_{\beta,r}\abs{\int_\MM\pp{\inddd{m(A_x)\odd} - \wt p_{\beta,r}(x)}f(x)\lambda(\d x)}^\alpha\frac{\d r}r}.
\eque
In particular, when $\alpha=2$, this is a family of centered Gaussian random variables with
\equh\label{eq:cov}
\cov(G_{2,\beta}(f),G_{2,\beta}(g)) =\int_\MM\int_\MM f(x)g(y)\Gamma\topp\beta(x,y)\lambda(\d x)\lambda(\d y), 
\eque
for all $f,g\in \calF_{\delta,\delta',\beta}(\MM)$, where
\equh\label{eq:sigma positive}
\Gamma\topp\beta(x,y) = \log \frac{\mu^\beta(A_x)+\mu^\beta(A_y)}{\dd^\beta(x,y)} = 4\int_0^\infty \frac{\Gamma_r\topp\beta(x,y)}r\d r,
\eque
with
\begin{multline}\label{eq:Gamma_r'}
\Gamma_r\topp\beta(x,y)\\ := \esp'_{\beta,r}\pp{\pp{\inddd{m(A_x)\odd} - \proba_{\beta,r}'\pp{m(A_x)\odd)}}\pp{\inddd{m(A_y)\odd} - \proba_{\beta,r}'\pp{m(A_y)\odd}}}.
\end{multline}
\end{theorem}
\begin{proof}
This time, the identity \eqref{eq:G M} again follows from the property of stochastic integrals, and we only need to show the finiteness of the right-hand side of \eqref{eq:G M}. Write
\equh\label{eq:Psi a f0}
\Psi_{\alpha,\beta,f}(r) = \esp'_{\beta,r}\abs{\int_\MM\pp{\inddd{m(A_x)\odd} - \wt p_r(x)}f(x)\lambda(\d x)}^\alpha.
\eque
The goal is to show
\equh\label{eq:Psi a f}
\int_0^\infty \frac{\Psi_{\alpha,\beta,f}(r)}r\d r<\infty.
\eque

We first deal with $\alpha=2$. 
We compute
\begin{align*}
\proba'_{\beta,r} & \pp{m(A_x) \odd, m(A_y)\odd} \\
& = \proba'_{\beta,r}\pp{m(A_x\setminus A_y)\odd}\proba'_{\beta,r}\pp{m(A_y\setminus A_x)\odd}\proba'_{\beta,r}\pp{m(A_x\cap A_y)\even}\\
& \quad+ \proba'_{\beta,r}\pp{m(A_x\setminus A_y)\even}\proba'_{\beta,r}\pp{m(A_y\setminus A_x)\even}\proba'_{\beta,r}\pp{m(A_x\cap A_y)\odd}\\
& = \esp'_{\beta,r}\pp{\frac14\pp{1-e^{-2\calS_{\beta,r}\mu(A_x)}-e^{-2\calS_{\beta,r}\mu(A_y)}+e^{-2\calS_{\beta,r}\mu(A_x\Delta A_y)}}}\\
& = \frac14\pp{1-e^{-(2\mu(A_x))^\beta r}-e^{-(2\mu(A_y))^\beta r}+e^{-(2\mu(A_x\Delta A_y))^\beta r}}.
\end{align*}
Then, taking $\Gamma\topp\beta_r(x,y)$ as given in \eqref{eq:Gamma_r'}, we have
\begin{align}\label{eq:Gamma_r''}
\Gamma_r\topp\beta(x,y) &= \proba'_{\beta,r}  \pp{m(A_x) \odd, m(A_y)\odd}  - \proba'_{\beta,r}\pp{m(A_x) \odd}  \proba'_{\beta,r}\pp{ m(A_y)\odd}\\
& =   \frac14\pp{e^{-(2\mu(A_x\Delta A_y))^\beta r} - e^{-((2\mu(A_x))^\beta+(2\mu(A_y))^\beta)r}}\ge 0.\nonumber
\end{align}
Therefore, we have
\begin{align}
& \Psi_{2,\beta,f} (r)\nonumber\\
&  = \esp'_{\beta,r}\int_\MM\int_\MM\pp{\inddd{m(A_x)\odd} - \wt p_{\beta,r}(A_x)}\pp{\inddd{m(A_y)\odd} - \wt p_{\beta,r}(A_y)}f(x)f(y)\lambda(\d x)\lambda(\d y)\nonumber\\
&
= \int_\MM\int_\MM \Gamma_r\topp\beta(x,y)f(x)f(y)\lambda(\d x)\lambda(\d y)\nonumber\\
& = \frac14\int_\MM\int_\MM\pp{e^{-(2\mu(A_x\Delta A_y))^\beta r} - e^{-((2\mu(A_x))^\beta+(2\mu(A_y))^\beta)r}}f(x)f(y)\lambda(\d x)\lambda(\d y). \nonumber
\end{align}
The first equality is the definition, the second follows from Fubini's theorem (since $\nn f_1<\infty$), and the third follows from \eqref{eq:Gamma_r''}.

Next, we prove
\equh\label{eq:Psi 2 f}
\Psi_{2,\beta,f}(r) \le C(r^{\delta}\wedge r^{-\delta'}) \mfa r>0,
\eque
for all $f\in \calF_{\delta,\delta',\beta}(\MM)$ (the constant $C$ depends on $f$). 
For the sake of simplicity, assume that $f\ge 0$. 
We break the outer integral into two parts and proceed as follows. First, 
\begin{align}
\int_{\mu(A_x)> 1/r}&  \int_\MM f(x)f(y) \pp{e^{-(2\mu(A_x\Delta A_y))^\beta r} - e^{-((2\mu(A_x))^\beta+(2\mu(A_y))^\beta)r}}\lambda(\d x)\lambda(\d y)\nonumber\\
 & \le \int_{\mu(A_x)>1/r}\int_\MM f(x)f(y)\lambda(\d x)\lambda(\d y) = \int_{\mu(A_x)>1/r}f(x)\lambda(\d x)\nn f_1 \nonumber\\
 & \le r^\delta \nn f_1\int_{\mu(A_x)>1/r}\mu(A_x)^\delta f(x)\lambda(\d x) \le C \nn f_1 \nn f_{1,\delta}r^\delta.\nonumber
\end{align}

Second,
using
\begin{align*}
e^{-(2\mu(A_x\Delta A_y))^\beta r} - e^{-((2\mu(A_x))^\beta+(2\mu(A_y))^\beta)r}& \le 1- e^{-2^\beta(\mu^\beta(A_x)+\mu^\beta(A_y)-\mu^\beta(A_x\Delta A_y))r}\\
& \le 2^\beta(\mu^\beta(A_x)+\mu^\beta(A_y)-\mu^\beta(A_x\Delta A_y))r \\
&\le 2^{\beta+1}\mu^\beta(A_x)r,
\end{align*}
we have
\begin{align}
\int_{\mu(A_x)\le 1/r}&  \int_\MM f(x)f(y) \pp{e^{-(2\mu(A_x\Delta A_y))^\beta r} - e^{-((2\mu(A_x))^\beta+(2\mu(A_y))^\beta)r}}\lambda(\d x)\lambda(\d y)\nonumber\\
 &\le \nn f_1\int_{\mu(A_x)\le 1/r}2^{\beta+1}r\mu^\beta (A_x)f(x)\lambda(\d x) \nonumber\\
 &\le 
 Cr\int_{\mu(A_x)\le 1/r}(1/r)^{\beta(1-\wt\delta)}\mu^{\beta\wt\delta}(A_x) f(x)\lambda(\d x)\le C r^{1-\beta+\beta\wt\delta}\nn f_{1,\beta\wt\delta},\nonumber
\end{align}
for $\wt\delta\in(0,1)$. Taking $\wt\delta = \delta/\beta$, the above is bounded by $Cr^{1-\beta+\delta}\le C r^\delta$ for $r\in[0,1]$. We have thus proved $\Psi_{2,\beta,f}(r)\le C r^\delta$ for $r\in[0,1]$. 
Alternatively,  using $e^{-(2\mu(A_x\Delta A_y))^\beta r} - e^{-((2\mu(A_x))^\beta+(2\mu(A_y))^\beta)r} \le  e^{-2r\mu(A_x\Delta A_y)} = e^{-2^\beta\dd^\beta(x,y)r}$, we have
\[
|\Psi_{2,\beta,f}(r)|  \le \int_\MM\int_\MM e^{-2^\beta\dd^\beta(x,y) r}f(x)f(y)\lambda(\d y)\lambda(\d x)\le Cr^{-\delta'},
\]
where the last step follows from (iii) of  Assumption \ref{assump:2}. We have thus proved \eqref{eq:Psi 2 f}.

Now \eqref{eq:Psi a f} follows with $\alpha=2$. For $\alpha\in(0,2)$, using $\Psi_{\alpha,\beta,f}(r)\le \Psi_{2,\beta,f}(r)^{\alpha/2}$, \eqref{eq:Psi a f} remains to hold. We have thus proved that \eqref{eq:G M} holds and the right-hand side is finite.

In particular for the case $\alpha=2$,  \eqref{eq:G M}, \eqref{eq:Gamma_r''} and Gamma integral yield
\begin{align*}
\cov(G_{2,\beta}(f),G_{2,\beta}(f)) & = 4 \int_0^\infty\int_\MM\int_\MM f(x)f(y)\Gamma_r\topp\beta(x,y)\lambda(\d x)\lambda(\d y)\frac{\d r}r\\
& =\int_\MM\int_\MM f(x)f(y)\log\frac{\mu^\beta(A_x)+\mu^\beta(A_y)}{\mu^\beta(A_x\Delta A_y)}\lambda(\d x)\lambda(\d y),
\end{align*}
where we used Frullani integral again. 
By Assumption \ref{assump:1}, $\mu(A_x\Delta A_y) = \dd(x,y)$. The stated formula \eqref{eq:cov} for the covariance function now follows.
\end{proof}

\subsection{A central limit theorem}
We consider another aggregated model combining ideas from \citep{enriquez04simple} and \citep{fu20stable} that scales to the generalized process just introduced. We only illustrate the case $H = 1/2$. 

Fix $\alpha\in(0,2]$. 
Let $\calX\topp\alpha$ be a symmetric $\alpha$-stable random variable determined by 
\[
\esp \exp\pp{\i \theta\calX\topp\alpha} = \exp\pp{-|\theta|^\alpha}. 
\]
Let $q$ be a uniform random variable over $(0,1)$ and given $q$ and $n\in\N$, let $N_n\topp q$ denote a Poisson point process on $(E,\calE)$ with intensity measure $nq\cdot\mu$. Assume $\calX\topp\alpha$ is independent from $q$ and $N_n\topp q$. Let $(\calX_i\topp\alpha,q_i,N_{n,i}\topp i)_{i\in\N}$ be i.i.d.~copies of $(\calX\topp\alpha,q,N_n\topp q)$. 
Set
\[
G_{\alpha,n}(f):=\frac1{m_n^{1/\alpha}}\summ i1{m_n}\frac{\calX_i\topp\alpha}{q_i^{1/\alpha}}\int_\MM \pp{\inddd{N_{n,i}\topp{q_i}(A_x)\odd} - \proba\pp{N_{n,i}\topp{q_i}(A_x)\odd\mmid q_i}}f(x)\lambda(\d x),
\]
for all $f\in\calF_{\delta,\delta',1}(\MM)$. This represents an aggregated model based on the triangular arrays of i.i.d.~point processes $(N_{n,i}\topp {q_i})_{n\in\N,i\in\N}$. 
\begin{theorem}
Under the assumption $m_n\gg n$, we have
\[
\limn\esp \exp\pp{\i G_{\alpha,n}(f)} = \esp \exp\pp{\i G_\alpha(f)}, \mfa f\in \calF_{\delta,\delta',1}(\MM).
\]
\end{theorem}
\begin{proof}
Introduce $N_n\topp r$ with $r\in(0,1)$ as a fixed number similarly as $N_n\topp q$. Write $\wt p_{n,r}(A) = (1/2)(1-e^{-2nr\mu(A)})$. We begin by writing
\begin{align*}
\esp e^{\i G_{\alpha,n}(f)} & = \pp{\esp \exp\pp{\frac{\i}{m_n^{1/\alpha}}\frac{\calX\topp\alpha}{q^{1/\alpha}}\int_\MM\pp{\inddd{N_n\topp{q}(A_x)\odd} - \wt p_{n,q}(A_x)}f(x)\lambda(\d x)}}^{m_n}\\
& = \pp{\int_0^1\esp \exp\pp{-\frac1{m_n r}\abs{\int_\MM\pp{\inddd{N_n\topp {r}(A_x)\odd}-\wt p_{n,r}(A_x)}f(x)\lambda(\d x)}^\alpha}\d r}^{m_n}\\
& = \pp{\int_0^n\esp \exp\pp{-\frac n{m_n r}\abs{\int_\MM\pp{\inddd{N_1\topp {r}(A_x)\odd}-\wt p_{1,r}(A_x)}f(x)\lambda(\d x)}^\alpha}\frac{\d r}n}^{m_n},
\end{align*}
where in the last step we first applied change of variables $r\mapsto nr$ and then used the fact that $N_n\topp{r/n}\eqd N_1\topp r$ and $\wt p_{n,r/n}(x) = \wt p_{1,r}(x) = \wt p_r(x)$. That is, with
\[
\what\Psi_{\alpha,f}(r) = \abs{\int_\MM\pp{\inddd{N_1\topp {r}(A_x)\odd}-\wt p_{1,r}(A_x)}f(x)\lambda(\d x)}^\alpha,
\]
we have  shown
\[
\esp \exp\pp{\i G_{\alpha,n}(f)} = \pp{1-\int_0^n{1-\esp\exp\pp{-\frac n{m_nr}\what \Psi_{\alpha,f}(r)}}\frac{\d r}n}^{m_n}.
\]
It suffices to show
\equh\label{eq:limit M}
\int_0^n{1-\esp\exp\pp{-\frac n{m_nr}\what \Psi_{\alpha,f}(r)}}\frac{\d r}n \sim \frac1{m_n}\int_0^\infty \frac{ \esp \what \Psi_{\alpha,f}(r)}r\d r
\eque
as $n\to\infty$. 

Fix $\epsilon\in(0,1)$. First we have 
\equh\label{eq:limit M1}
\int_\epsilon^{\epsilon\inv}{1-\esp\exp\pp{-\frac n{m_nr}\what \Psi_{\alpha,f}(r)}}\frac{\d r}n \sim \frac1{m_n}\int_\epsilon^{\epsilon\inv}\frac{\esp \what\Psi_{\alpha,f}(r)}r\d r.
\eque
This step follows from the bounded convergence theorem with the observation $|\what\Psi_{\alpha,f}(r)|\le (\int_\MM |f(x)|\lambda(\d x))^\alpha<\infty$ almost surely and the assumption $m_n\gg n$. 
Second, we have
\equh\label{eq:limit M2}
\int_{[0,n]\setminus[\epsilon,\epsilon\inv]}{1-\esp\exp\pp{-\frac n{m_nr}\what \Psi_{\alpha,f}(r)}}\frac{\d r}n \le \frac1{m_n}\int_{[0,\epsilon]\cup[\epsilon\inv,\infty)}\frac{\esp\what\Psi_{\alpha,f}(r)}r\d r,
\eque
and (recall $\Psi_{\alpha,1,f}(r)$ in \eqref{eq:Psi a f0})
\equh\label{eq:limit M3}
\esp\what\Psi_{\alpha,f}(r) =  \Psi_{\alpha,1,f}(r) \le  \pp{\Psi_{2,1,f}(r)}^{\alpha/2}\le C(r^\delta\wedge r^{-\delta'})^{\alpha/2},
\eque
for some constant $C$ depending on $f\in\calF_{\delta,\delta',1}(\MM)$, where the last step was shown in \eqref{eq:Psi 2 f}. Now, combining \eqref{eq:limit M1}, \eqref{eq:limit M2} and \eqref{eq:limit M3} and letting $\epsilon\downarrow 0$, the above estimates yield \eqref{eq:limit M}.
\end{proof}
\appendix
\section{An example by Iksanov and Kotelnikova}\label{sec:Sasha}
\citet{iksanov22small} considered a limit theorem for the Karlin model in the regime that the sampling frequencies $(p_j)_{j\in\N}$ decay faster than any polynomial rate (recall that for the Karlin model with $\alpha\in(0,1)$, $p_j\sim Cj^{-1/\alpha}$). Namely, they considered the case that for $\rho(t):=|\{k\in\N:p_k\ge 1/t\}|$, that there exists a slowly varying function $g$ at infinity with $g(x)\to\infty$ as $x\to\infty$ such that $\lim_{x\to\infty}(\rho(\lambda t)- \rho(t))/g(t)= \log\lambda$ for all $\lambda>0$. 
We omit the description of the model (they actually investigated a sophisticated hierarchical Karlin model and the discussions here concern only their results at the first level), but focus on the limit processes. 

In particular, they investigated the scaling limit of the $j$-occupancy process. That is, the counting process for the number of urns containing exactly $j$ after $n$ rounds. The limit process \citep[Eq.~(4)]{iksanov22small} is denoted by $(X_j(u))_{u\in\R}$ below. The parameterization is a little unconventional as the index set is $\R$, although we shall eventually see that applying a time-change $u\mapsto t = \log u\in(0,\infty)$ we can relate $(X_ j)_{j\in\N}$ to the process of interest here $G^{1/2}$.  

Stochastic-integral representations for their limit processes were provided in \citep{iksanov22small}. Some preparations are needed for the Gaussian random measures involved. Let $\proba_{U^\infty}$ denote the uniform probability measure on $[0,1]^\N$ (the distribution of i.i.d.~uniform random variables).  Let $W_\infty$ denote a Gaussian random measure on $\R\times[0,1]^\N$ with intensity measure $\d x\d \proba_{U^\infty}$. For each $j\in\N$, let $W_j$ denote $W_\infty$ restricted to $\R\times[0,1]^j$, that is,
\[
W_j(\d x,\d y_1,\dots,\d y_j) :=W_\infty\pp{\d x,\d y_1,\dots,\d y_j, [0,1]^\N}, x\in\R, y_i\in[0,1], i=1,\dots,j.
\]
Then, the limit $j$-occupancy processes  have the following stochastic-integral representation
\begin{multline}\label{eq:X}
\pp{X_j(u)}_{u\in\R,j\in\N}\\
 = \pp{\int_{\R\times[0,1]^{j+1}}\pp{\inddd{y_1\cdots y_{j+1}<\psi_0(e^{-(x-u)})<y_1\cdots y_j} - \psi_j(e^{-(x-u)})}W_{j+1}(\d x\d y_1\cdots \d y_{j+1})}_{u\in\R,j\in\N},
\end{multline}
where $\psi_j(x) = \proba(N_x = j)$ and $N_x$ is a Poisson random variable with parameter $x>0$ and $j\in\N_0$. 

The processes $(X_j)_{j\in\N}$ also have the following representation. Let $(\Omega',\calF',\proba')$ denote a probability space and $M$ a Gaussian random measure on $(\Omega',\calF')$ with control measure $r\inv \d r\d\proba'$, and $N'$ is a standard Poisson process on $(\Omega',\proba')$. 
\begin{lemma}
We have the following equivalent representation of \eqref{eq:X}:
\[
\pp{X_j(\log t)}_{t\in(0,\infty),j\in\N} \eqd \pp{\int_{\R\times\Omega'}\pp{\inddd{N'(rt) = j} - \proba(N'(rt) = j)}M(\d r \d\omega')}_{t\in(0,\infty), j\in\N}.
\]
\end{lemma}
\begin{proof}
We first re-write the right-hand side of \eqref{eq:X} in a more probabilistic way. We can view $\wt\Omega = [0,1]^\N$ as a probability space and then $Y_j(\wt\omega) \equiv Y_j((y_k)_{k\in\N}) := y_j$ as i.i.d.~uniform random variables with respect to the law $\proba_{U^\infty}$. Note also that then $E_j:=-\log Y_j$ are i.i.d.~standard exponential random variables. Then, the event in the indicator function in the integrand of the right-hand side of \eqref{eq:X} becomes
\begin{align*}
\ccbb{\prodd k1{j+1}Y_k<\exp\pp{-e^{-(x-u)}}<\prodd k1j Y_k} & = \ccbb{\summ i1{j+1}(-\log Y_i)>e^{-(x-u)}>\summ i1j (-\log Y_i)}\\
 & = \ccbb{\summ i1j E_i<e^{-(x-u)}<\summ i1{j+1}E_i} = \ccbb{N(e^{-(x-u)}) = j},
\end{align*}
where we also set $N(t) = \max\{j\in\N: \summ i1j E_i\le t\}$ and $N(t) = 0$ if the maximum is over an empty set so that $(N(t))_{t\ge 0}$ with respect to $\proba_{U^\infty}$ is a standard Poisson process. That is, we have shown that \eqref{eq:X} can be re-written as
\begin{multline*}
\pp{X_j(u)}_{u\in\R,j\in\N}\\ = \pp{\int_{\R\times[0,1]^\N}\pp{\inddd{N(e^{-(x-u)}) = j} - \proba_{U_\infty}\pp{N\pp{e^{-(x-u)}} = j}}W_\infty(\d x \d \vv y)}_{u\in\R,j\in\N}.
\end{multline*}
Applying the change of variables $r = e^{-x}$  (so $\d x = r\inv\d r$ and the corresponding domain becomes $r\in(0,\infty)$ and also $t = e^u$, we have proved the stated result.  
\end{proof}

Then, with $\wt X_j(t) = X_j(\log t), t\in(0,\infty)$,  $\sif\ell1 \wt X_{2j-1}(t)$ is formally (since this is not point-wise convergent) our generalized Gaussian process: compare with $G^{1/2}(f)$ in \eqref{eq:G(f)} (up to a multiplicative constant as the control measures differ by a multiplicative constant $2$). 

Processes with similar representations often arise as scaling limits of counting processes associated to models with certain renewal structures. See \citet{alsmeyer25decoupled} for a limit theorem on the counting process for a family of decoupled renewal processes. 

\def\cprime{$'$} \def\polhk#1{\setbox0=\hbox{#1}{\ooalign{\hidewidth
  \lower1.5ex\hbox{`}\hidewidth\crcr\unhbox0}}}
  \def\polhk#1{\setbox0=\hbox{#1}{\ooalign{\hidewidth
  \lower1.5ex\hbox{`}\hidewidth\crcr\unhbox0}}}

\end{document}